\documentclass[11pt]{article}
\usepackage{amsmath,amssymb,amsthm, latexsym,color,epsfig,enumerate,a4, graphicx}

\theoremstyle{plain}
\newtheorem{theorem}{Theorem}[section]
\newtheorem{lemma}[theorem]{Lemma}
\newtheorem{claim}[theorem]{Claim}

\newtheorem{definition}[theorem]{Definition}

\newcommand{\eps}{\ensuremath{\varepsilon}}
\newcommand{\Gnp}{\ensuremath{\mathcal \mathcal G(n,p)}}

\date{}
\title{\vspace{-0.7cm}Almost-spanning universality in random graphs}
\author{
David Conlon\thanks{Mathematical Institute, Oxford OX2 6GG, United
Kingdom. Email: {\tt david.conlon@maths.ox.ac.uk}. Research
supported by a Royal Society University Research Fellowship.} \and
Asaf Ferber\thanks{Department of Mathematics, Yale University and
Department of Mathematics, MIT. Email: {\tt asaf.ferber@yale.edu,
ferbera@mit.edu}.} \and Rajko Nenadov$^\ddag$ \and Nemanja \v
Skori\'c\thanks{Institute of Theoretical Computer Science, ETH
Z\"urich, 8092 Z\"urich, Switzerland. \newline Email: {\tt
\{rnenadov|nskoric\}@inf.ethz.ch}.} }

\begin{document}
\maketitle

\begin{abstract}
A graph $G$ is said to be $\mathcal H(n,\Delta)$-universal if it
contains every graph on at most $n$ vertices with maximum degree at most
$\Delta$. It is known that for any $\varepsilon > 0$ and any natural
number $\Delta$ there exists $c > 0$ such that the random graph $G(n,p)$ 
is asymptotically almost surely $\mathcal
H((1-\varepsilon)n,\Delta)$-universal for $p \geq c (\log
n/n)^{1/\Delta}$. Bypassing this natural boundary, we show that for $\Delta \geq 3$ the same
conclusion holds when $p \gg n^{-\frac{1}{\Delta-1}}\log^5 n$.
\end{abstract}

\section{Introduction}

Given a family of graphs $\mathcal H$, a graph $G$ is said to be $\mathcal H$-universal if it contains every member of $\mathcal H$ as a subgraph (not necessarily induced). 
Universal graphs have been studied quite extensively, particularly with respect to families of forests, planar graphs and graphs of bounded degree (see, for example, \cite{AC08, ACKRRS, BLC, BCS, BCLR, CG, DKRR, FNP, KL} and their references). In particular, it is of interest to find sparse universal graphs.

Let $\mathcal H(n, \Delta)$ be the family of all graphs on at most $n$ vertices with maximum degree at most $\Delta$. Building on earlier work with several authors \cite{AA, ACKRRS, ACKRRS01}, Alon and Capalbo \cite{AC, AC08} showed that there are graphs with at most $c_{\Delta} n^{2 - 2/\Delta}$ edges which are $\mathcal H(n, \Delta)$-universal. A simple counting argument shows that this result is best possible.

The construction of Alon and Capalbo is explicit. An earlier approach had been to study whether random graphs could be $\mathcal H(n, \Delta)$-universal. The binomial random graph $\Gnp$ is the graph formed by choosing every edge of a graph on $n$ vertices independently with probability $p$. We say that $\Gnp$ satisfies a property $\mathcal{P}$ asymptotically almost surely (a.a.s.) if $\Pr\left[\Gnp \in \mathcal{P}\right]$ tends to $1$ as $n$ tends to infinity. The first result on universality in random graphs was proved by Alon, Capalbo, Kohayakawa, R\"odl, Ruci\'nski and Szemer\'edi \cite{ACKRRS}, who showed that for any $\varepsilon > 0$ and any natural number $\Delta$ there exists a constant $c > 0$ such that the random graph $\Gnp$ is a.a.s.~$\mathcal H((1-\varepsilon)n, \Delta)$-universal for $p \geq c (\log n/n)^{1/\Delta}$. 

In this theorem, some slack is allowed by only asking that the random graph contains subgraphs with $(1 - \varepsilon)n$ vertices. However, one can also ask whether the random graph $\Gnp$ contains all subgraphs of maximum degree $\Delta$ with exactly $n$ vertices, that is, whether it is $\mathcal H(n, \Delta)$-universal. Because we no longer have any extra room to manoeuvre, this problem is substantially more difficult to treat than the almost-spanning version. Nevertheless, Dellamonica, Kohayakawa, R\"odl and Ruci\'nski \cite{DKRR} have shown that for any natural number $\Delta \geq 3$ there exists a constant $c > 0$ such that $\Gnp$ is a.a.s.~$\mathcal H(n, \Delta)$-universal for $p \geq c (\log n/n)^{1/\Delta}$. The case $\Delta=2$ was later treated by Kim and Lee~\cite{KL}, who obtained similar bounds to those in \cite{DKRR}. 

These results on embedding large bounded-degree graphs in the random graph have proved useful in other contexts \cite{BKT, KRSS}. To give an example, we define the size Ramsey number $\hat{r}(H)$ of a graph $H$ to be the smallest number of edges $m$ in a graph $G$ which is Ramsey with respect to $H$, that is, such that any 2-colouring of the edges of $G$ contains a monochromatic copy of $H$. A famous result of Beck \cite{Be83} states that the size Ramsey number of the path $P_n$ with $n$ vertices is at most $c n$ for some fixed constant $c > 0$. An extension of this result to graphs of maximum degree $\Delta$ was recently given by Kohayakawa, R\"odl, Schacht and Szemer\'edi \cite{KRSS}, who showed that there is a constant $c > 0$ depending only on $\Delta$ such that if $H$ is a graph with $n$ vertices and maximum degree $\Delta$ then
\[\hat{r}(H) \leq c n^{2 - 1/\Delta} (\log n)^{1/\Delta}.\]
A key component in their proof is an embedding lemma like that used in~\cite{ACKRRS}. One could therefore hope to improve this bound by improving the bounds for universality in random graphs.

Here we make some initial progress on these problems by improving the theorem of Alon, Capalbo, Kohayakawa, R\"odl, Ruci\'nski and Szemer\'edi \cite{ACKRRS} on  almost-spanning universality as follows.

\begin{theorem} \label{thm:main_universal}
For any constant $\eps > 0$ and integer $\Delta \ge 3$, the random graph $\Gnp$ is a.a.s.\ universal for the family $\mathcal{H}((1 -
\eps)n, \Delta)$, provided that $p \gg n^{-\frac{1}{\Delta - 1}} \log^5 n$.
\end{theorem}

This result bypasses a natural barrier, since $(\log
n/n)^{1/\Delta}$ is roughly the lowest probability at which we can expect that every
collection of $\Delta$ vertices will have many neighbors in common,
a condition which is extremely useful if one wishes to embed graphs
of maximum degree $\Delta$. On the other hand, the lowest probability at which one might hope that the
random graph $\Gnp$ is a.a.s.~$\mathcal H((1- \varepsilon)n, \Delta)$-universal is $n^{- 2/(\Delta+1)}$.
Indeed, below this probability, $\Gnp$ will typically not contain $(1 - \varepsilon) \frac{n}{\Delta +
1}$ vertex-disjoint copies of $K_{\Delta + 1}$ (see, for example,~\cite{JLR}). Thus, for $\Delta = 3$ our result is optimal up to the logarithmic factor, while for $\Delta \ge 4$ the gap remains.

In proving Theorem~\ref{thm:main_universal}, we will make use of a recent result of Ferber, Nenadov and Peter~\cite{FNP} which improves the bounds in~\cite{ACKRRS01} and \cite{DKRR} for bounded-degree graphs which satisfy certain density and partitioning properties (see Section~\ref{sec:FNP}). When embedding a graph $H \in \mathcal H((1 - \varepsilon)n, \Delta)$, we will first find a subgraph $H' \subseteq H$ suitable for the application of the main result of \cite{FNP} by removing all small components and certain short cycles in $H$. We then embed $H'$, after which we place the deleted pieces in an appropriate way so as to obtain an embedding of $H$.

\subsection{Notation}
For a graph $G = (V, E)$, we denote by $v(G)$ and $e(G)$ the size of the vertex and edge sets, respectively. For a vertex $v \in V$, we write $\Gamma_G^{(i)}(v) := \{w \in V \, : \, \mathrm{dist}(v, w) = i\}$ for the set of vertices at distance exactly $i$ from $v$. For simplicity, we let $\Gamma_G^{(0)}(v) := \{v\}$ and $\Gamma_G(v) := \Gamma_G^{(1)}(v)$. Furthermore, for a set $S \subseteq V$, we define $\Gamma_G^{(i)}(S) :=
\{w \in V \, : \,  \min_{v \in S} \mathrm{dist}(v, w) = i\}$.
Similarly, we let $B_G^{(i)}(v) := \bigcup_{j = 0}^i \Gamma_G^{(j)}(v)$ be the \emph{ball} of radius $i$ around $v$ in $G$, i.e., the set of all vertices at distance at most $i$ from $v$. For an integer $k$ and a set of vertices $S \subseteq V$, we say that $S$ is \emph{$k$-independent} if $B^{(k)}_G(v) \cap (S \setminus \{v\}) = \emptyset$ for every $v \in S$, i.e., every two vertices in $S$ are at distance at least $k+1$ in $G$. If there is no risk of ambiguity, we omit $G$ from the subscript.

Given a (hyper)graph $H = (V, E)$, we say that two edges $e_1, e_2 \in E$ are independent if $e_1 \cap e_2 = \emptyset$. For two graphs $G'$ and $G''$, we write $G' \cong G''$ if they are isomorphic. A family of subsets $A_1, \ldots, A_k \subseteq V$, for some integer $k$, is a partition of $V$ if $A_i \cap A_j = \emptyset$ for all distinct $i, j \in \{1, \ldots, k\}$ and $V = \bigcup_{i = 1}^k A_i$. Note that we allow subsets $A_i$ to be empty. Finally, for an integer $k \in \mathbb{N}$, we use the standard notation $[k] := \{1, \ldots, k\}$.

We use the standard asymptotic notation $O$, $o$, $\Omega$ and $\omega$. Furthermore, given two functions $a$ and $b$, we write $a \ll b$ if $a = o(b)$ and $a \gg b$ if $a = \omega(b)$.

\section{Tools and preliminaries}

In this section, we present some tools to be used in the proof of our
main result.

\subsection{Probabilistic tools}

We will use lower tail estimates for random
variables which count the number of copies of certain graphs in a
random graph. The following version of
Janson's inequality, tailored for graphs, will suffice. This
statement follows immediately from Theorems $8.1.1$ and $8.1.2$ in
\cite{alon2004probabilistic}.

\begin{theorem}[Janson's inequality] \label{thm:Janson}
Let $p \in (0, 1)$ and consider a family $\{ H_i \}_{i \in
\mathcal{I}}$ of subgraphs of the complete graph on the vertex set
$[n]$. Let $G \sim \Gnp$. For each $i \in \mathcal{I}$, let $X_i$
denote the indicator random variable for the event that $H_i \subseteq
G$ and, for each ordered pair $(i, j) \in \mathcal{I} \times
\mathcal{I}$ with $i \neq j$, write $H_i \sim H_j$ if $E(H_i)\cap
E(H_j)\neq \emptyset$. Then, for
\begin{align*}
X &= \sum_{i \in \mathcal{I}} X_i, \\
\mu &= \mathbb{E}[X] = \sum_{i \in \mathcal{I}} p^{e(H_i)}, \\
\delta &= \sum_{\substack{(i, j) \in \mathcal{I} \times \mathcal{I} \\ H_i \sim H_j}} \mathbb{E}[X_i X_j] = \sum_{\substack{(i, j) \in \mathcal{I} \times \mathcal{I} \\ H_i \sim H_j}} p^{e(H_i) + e(H_j) - e(H_i \cap H_j)}
\end{align*}
and any $0 < \gamma < 1$,
$$ \Pr[X < (1 - \gamma)\mu] \le e^{- \frac{\gamma^2 \mu^2}{2(\mu + \delta)}}. $$
\end{theorem}

\subsection{Universality for ``nicely partitionable" graphs} \label{sec:FNP}

In the following definition, we introduce a family of graphs that
admit a ``nice partition".

\begin{definition} \label{def:partition}
Let $n, d$ and $t$ be positive integers and let $\eps$ be a
positive number. The family of graphs $\mathcal{F}(n, t,
\eps, d)$ consists of all graphs $H$ on $n$ vertices for
which there exists a partition $W_0, \ldots, W_t$ of $V(H)$ such that
\begin{enumerate}[(i)]
\item $|W_t| = \lfloor \eps n \rfloor$,
\item $W_0 = \Gamma(W_t)$,
\item $W_t$ is $3$-independent,
\item $W_i$ is $2$-independent for every $1\leq i \leq t-1$, and
\item for every $1\leq i\leq t$ and for every $w\in W_i$, $w$ has at
most $d$ neighbors in $W_0\cup\ldots \cup W_{i-1}$.
\end{enumerate}
\end{definition}

The following result, due to Ferber, Nenadov and Peter
\cite{FNP}, shows that for an appropriate $p$ a
typical $G \sim \Gnp$ is $\mathcal{F}(n, t, \eps, d)$-universal.

\begin{theorem}[Theorem 4.1 in \cite{FNP}] \label{thm:universality_partition}
Let $n$ and $t = t(n)$ be positive integers, let $d=d(n)\geq2$ be an integer and let $\eps = \eps(n) < 1 / (2d)$. Then the random graph $\Gnp$ is a.a.s.~$\mathcal{F}(n, t, \eps, d)$-universal, provided that $p \gg \eps^{-1} t n^{-1/d} \log^2 n$.
\end{theorem}

We remark that this result is actually stronger than we need, since it gives a statement about spanning graphs while in this paper we only deal with almost-spanning graphs. Thus, when applying Theorem~\ref{thm:universality_partition}, we will make our graph spanning by adding a certain number of isolated vertices.

\subsection{Universality for graphs of small size} \label{sec:components}




In this section, we prove auxiliary lemmas which will allow us to ignore small components in the proof of Theorem \ref{thm:main_universal} (see Phase III in the proof of Theorem \ref{thm:main_universal}).

\begin{lemma} \label{lemma:small_universal_single}
Let $\Delta \ge 3$ and $k$ be integers and let $H \in \mathcal{H}(\log^k n, \Delta)$ be a connected graph with $v(H) \ge \Delta + 2$. Then $G \sim \Gnp$  contains $H$ with probability $1 - e^{-\omega(n)}$, provided that $p \gg n^{-2/(\Delta + 1)}$.
\end{lemma}
\begin{proof}

Let $(h_1, \ldots, h_{v(H)})$ be an arbitrary ordering of the vertices of $H$ and let $V_1, \ldots, V_{v(H)} \subseteq V(G)$ be disjoint subsets of order $n / \log^k n$. We wish to use Janson's inequality to prove the lemma. For that, we will restrict our attention to the ``canonical" copies of $H$: the family $\{H_i\}_{i \in \mathcal{I}}$ consists of all those copies of $H$ in $K_n$ with the property that the vertex $h_j$ belongs to the set $V_j$ for every $j \in \{1, \ldots, v(H)\}$. 
 We now estimate the parameters $\mu$ and $\delta$ defined in Theorem \ref{thm:Janson}.

Let $X$ be the number of copies $H_i$, $i \in \mathcal{I}$, that appear in $G$. Then $\mu = \mathbb{E}[X]$ satisfies the bound
$$ \mu = \left( \frac{n}{\log^k n} \right)^{v(H)} p^{e(H)} \ge \left( \frac{n}{\log^k n} \right)^{v(H)} p^{v(H) \Delta / 2} \gg \left( \frac{n}{\log^k n} \cdot n^{-\Delta/(\Delta+1)} \right)^{v(H)} \gg n, $$
where the last inequality follows from $v(H) \ge \Delta + 2$. 

Next, note that for any proper subgraph $J \subset H$ we have
\begin{itemize}
    \item $e(J) \le \binom{v(J)}{2}$ if $v(J) \le \Delta$, and
    \item $e(J) \le \frac{1}{2}((v(J) - 1)\Delta + \Delta - 1) = \frac{1}{2}(v(J) \Delta - 1)$ otherwise.
\end{itemize}
The second estimate follows from the fact that $H$ is connected and therefore there exists at least one vertex in $J$ with degree at most $\Delta - 1$. We can now rewrite $\delta$ as
$$ \delta = \sum_{\substack{(i, j) \in \mathcal{I} \times \mathcal{I} \\ H_i \sim H_j}} p^{e(H_i) + e(H_j) - e(H_i \cap H_j)} = \sum_{J \subset H} \sum_{H_i \cap H_j \cong J} p^{2e(H) - e(J)}.$$
Using the observations above, we split the sum based on the
size of $v(J)$, getting
 $$\delta \le \sum_{\substack{J \subset H\\v(J) \le \Delta}} \sum_{H_i \cap H_j \cong J} p^{2 e(H) - \binom{v(J)}{2}} + \sum_{\substack{J \subset H\\v(J) \ge \Delta + 1}} \sum_{H_i \cap H_j \cong J} p^{2 e(H) - (v(J)\Delta - 1)/2}. $$
We can bound the number of summands by first deciding on an embedding of $J$, which can be done in $(n / \log^k n)^{v(J)}$ ways, and then on an embedding of the remaining parts of the two copies of $H$ which intersect on $J$ (at most $(n / \log^k n)^{2(v(H) - v(J))}$ ways), yielding
$$ \delta \le \sum_{j = 2}^{\Delta} \binom{v(H)}{j} \left(\frac{n}{\log^k n}\right)^{2v(H) - j} p^{2e(H) - \binom{j}{2}} + \sum_{j = \Delta + 1}^{v(H) - 1} \binom{v(H)}{j} \left(\frac{n}{\log^k n}\right)^{2v(H) - j} p^{2e(H) - (j\Delta - 1)/2}.
$$
Finally, by pulling $\mu^2$ outside, we obtain
$$ \delta \le \mu^2 \left(  \sum_{j = 2}^{\Delta} \binom{v(H)}{j} \left( \frac{n}{\log^k n} \right)^{-j} p^{-\binom{j}{2}} + \sum_{j = \Delta + 1}^{v(H) - 1} \binom{v(H)}{j} \left( \frac{n}{\log^k n} \right)^{-j} p^{-(j\Delta - 1)/2} \right).$$
By substituting for $p$, we get the following upper bound on the first sum,
\begin{align*}
\delta_1 := \sum_{j = 2}^{\Delta} \binom{v(H)}{j} \left( \frac{n}{\log^k n} \right)^{-j} p^{-\binom{j}{2}} &\ll \sum_{j = 2}^{\Delta} (\log n)^{2kj} n^{-j} n^{\frac{j(j-1)}{\Delta + 1}} \le \sum_{j = 2}^\Delta (\log n)^{2\Delta k} n^{\frac{\Delta}{\Delta+1}(j-1) -j} \\
&= (\log n)^{2\Delta k} \sum_{j = 2}^{\Delta} n^{-\frac{\Delta + j}{\Delta + 1}} \le (\log n)^{2\Delta k} \Delta  n^{-1 - 1/(\Delta+1)} \ll 1/n.
\end{align*}
Proceeding similarly for the second sum, we have
$$ \delta_2 := \sum_{j = \Delta + 1}^{v(H) - 1} \binom{v(H)}{j} \left( \frac{n}{\log^k n} \right)^{-j} p^{-(j\Delta - 1)/2}  \ll \sum_{j = \Delta + 1}^{v(H) - 1} (\log n)^{2jk}  n^{-j} n^{\frac{j\Delta - 1}{\Delta +1}} = \sum_{j = \Delta + 1}^{v(H)-1} (\log n)^{2jk} n^{-\frac{j+1}{\Delta+1}}.$$
Since $j \ge \Delta + 1$ in the above sum, we have
$$ n^{-\frac{j+1}{\Delta+1}} = n^{-1 - \frac{j - \Delta}{\Delta + 1}} \ll (\log n)^{-2k(j+1)} n^{-1},$$
thus it easily follows that $\delta_2 = o(1/n)$. Summing up, we get $\delta = o(\mu^2/n)$.

Finally, by applying Theorem \ref{thm:Janson} with parameters $\mu$ and $\delta$, we obtain
$$ 
	\Pr[X < \mu/2] \le e^{-\mu^2/(8(\mu + \delta))}.
$$
Since $\mu \gg n$, this implies the conclusion of the lemma.
\end{proof}

The next lemma deals with graphs on at most $\Delta + 1$ vertices. Since we treat $\Delta$ as a constant, it is a standard application of Janson's inequality and we omit the proof.
\begin{lemma} \label{lemma:very_small}
Let $\Delta \ge 3$ be an integer and $H$ any graph on at most $\Delta + 1$ vertices. Then $G \sim \Gnp$ contains $H$ with probability $1 - e^{-\omega(n)}$, provided that $p \gg n^{-2/(\Delta + 1)}$.
\end{lemma}

Finally, we make use of Lemmas~\ref{lemma:small_universal_single} and  \ref{lemma:very_small}  to show that every large induced subgraph of $\Gnp$ contains all connected graphs from $\mathcal{H}(\log^k n, \Delta)$ simultaneously.

\begin{lemma} \label{lemma:small_universal}
Let $\eps > 0$ be a constant and $\Delta \ge 3$ and $k$ be integers. Then, for $p \gg n^{-2/(\Delta+1)}$, $G \sim \Gnp$ a.a.s.~has the following property: for every $V' \subseteq V(G)$ of order $|V'| \ge \eps n$, $G[V']$ contains every connected graph $H \in \mathcal{H}(\log^k n, \Delta)$.
\end{lemma}
\begin{proof}
Let $G \sim \Gnp$, with $p$ as stated in the lemma. By Lemmas~\ref{lemma:small_universal_single} and \ref{lemma:very_small}, for a fixed subset $V' \subseteq V(G)$ of order $|V'| \ge \eps n$ and a connected graph $H \in \mathcal{H}(\log^{k+1} (\eps n), \Delta)$, we have
$$ \Pr[H \subseteq G[V']] = 1 - e^{-\omega(n)}.$$
Note that as $ \log^{k}n \leq \log^{k+1} (\eps n)$, these estimates
also apply for every connected graph $H \in \mathcal{H}(\log^{k} n,
\Delta)$. Since there are at most $2^n$ choices for $V'$ and at most
$$ \sum_{v_H = 2}^{\log^k n} \sum_{e_H = 1}^{\Delta v_H / 2} \binom{v_H^2}{e_H} \le \Delta \log^{2k} n \cdot \binom{\log^{2k} n}{\Delta \log^k n} = o(2^n)$$
connected graphs $H \in \mathcal{H}(\log^k n, \Delta)$, an application of the
union bound completes the proof.
\end{proof}

\subsection{Systems of disjoint representatives in hypergraphs} \label{sec:hyper_match}

The following lemma will allow us to place a set of short cycles in the proof of Theorem \ref{thm:main_universal} (see Phase II in the proof of Theorem \ref{thm:main_universal}). We make no effort to optimize the logarithmic factor in the bound on the edge probability $p$.

\begin{lemma} \label{lemma:embedding_cycles}
Let $\eps > 0$ be a constant, $\Delta \ge 3$, $3 \le g \leq 2 \log n$ and $t \le \varepsilon n / (32 \log^3 n)$ be integers and let $D \subseteq [n]$ be a subset of size $\varepsilon n / (4\log n)$. Then $G \sim \Gnp$ satisfies the following with probability at least $1 - o(1/n)$, provided that $p \gg \left(\log^7 n / n \right)^{1/(\Delta - 1)}$: for any family of subsets $\{W_{i, j} \}_{(i, j) \in [t] \times [g]}$, where
\begin{enumerate}[(i)]
\item $W_{i,j} \subseteq V(G) \setminus D$ and $|W_{i,j}| = \Delta - 2$ for all $(i, j) \in [t] \times [g]$, and
\item $W_{i, j} \cap W_{i', j'} = \emptyset$ for all $i \neq i'$,
\end{enumerate}
there exists a family of cycles $\{ C_i = (c_{i_1}, \ldots, c_{i_g}) \}_{i \in [t]}$, each of length $g$, such that
\begin{enumerate}[(i)]
\item $V(C_i) \subseteq G[D]$ and $V(C_i) \cap V(C_{i'}) = \emptyset$, for all $i \neq i'$, and
\item $W_{i, j} \subseteq \Gamma_G(c_{i_j})$ for all $(i, j) \in [t] \times [g]$.
\end{enumerate}
\end{lemma}

Lemma \ref{lemma:embedding_cycles} will follow as a corollary of Lemma \ref{lemma:cycles} and the following generalization of Hall's matching criterion due to Aharoni and Haxell~\cite{aharoni2000hall}.

\begin{theorem}[Corollary 1.2, \cite{aharoni2000hall}] \label{thm:hyper_match}
Let $g$ be a positive integer and $\mathcal{H} = \{H_1, \ldots, H_t\}$ a family of $g$-uniform hypergraphs on the same vertex set. If, for every $\mathcal{I} \subseteq [t]$, the hypergraph $\bigcup_{i \in \mathcal{I}} H_i$ contains a matching of size greater than $g(|\mathcal I| - 1)$, then there exists a function $f:[t] \rightarrow \bigcup_{i = 1}^t E(H_i)$ such that $f(i) \in E(H_i)$ and $f(i) \cap f(j) = \emptyset$ for $i \neq j$.
\end{theorem}

The following lemma allows us to find the matchings required by Theorem \ref{thm:hyper_match} in a greedy way, i.e., edge by edge.

\begin{lemma} \label{lemma:cycles}
Let $\eps > 0$ be a constant, $\Delta \ge 3$, $3 \le g \leq 2 \log n$ and $k \le \varepsilon n / (32 \log^3 n)$ be integers and let $D \subseteq [n]$ be a subset of order $\varepsilon n / (4 \log n)$. Then $G \sim \Gnp$ satisfies the following with probability at least $1 - o(1/n^2)$, provided that $p \gg \left(\log^7 n / n \right)^{1/(\Delta - 1)}$: for any family of subsets $\{W_{i, j} \}_{(i, j) \in [k] \times [g]}$, where
\begin{enumerate}[(i)]
\item $W_{i,j} \subseteq V(G) \setminus D$ and $|W_{i,j}| = \Delta - 2$ for all $(i, j) \in [k] \times [g]$, and
\item $W_{i, j} \cap W_{i', j'} = \emptyset$ for all $i \neq i'$,
\end{enumerate}
 and any subset $D' \subseteq D$ of order $|D'| \geq |D| - g^2 k$, there exists $i \in \{1, \ldots, k\}$ and a cycle $C = (c_1, \ldots, c_g) \subseteq G[D']$ of length $g$ such that $W_{i, j} \subseteq \Gamma_G(c_j)$ for all $j \in \{1, \ldots, g\}$.
\end{lemma}
\begin{proof}
Our aim is to show that for a subset $D' \subseteq D$ and a family $\{W_{i,j}\}_{(i,j)\in[k]\times [g]}$ satisfying properties $(i)$ and $(ii)$, the graph $G \sim \Gnp$ fails to satisfy the conclusion of the lemma with probability at most $e^{-\omega(k \log^3 n)}$. Since we can choose the family $\{W_{i,j}\}_{(i,j)\in[k]\times[g]}$ in at most $\binom{n}{\Delta-2}^{k g} \le 2^{(\Delta - 2) kg \log n} = 2^{o(k \log^3 n)}$ ways and $D'$ in at most $\binom{n}{g^2 k} \le 2^{g^2 k \log n} \le 2^{4 k \log^3 n}$ ways, the lemma follows by a simple application of the union bound. It remains to prove the desired bound on the probability of a failure.

We first introduce some notation. Given a cycle $C = (c_1, \ldots, c_g) \subseteq K_n$ of length $g$, we define the graph $C \oplus i$ by
\begin{align*}
V(C \oplus i) = V(C) \cup \bigcup_{j = 1}^g W_{i,j}, \quad \text{and} \quad
E(C \oplus i) = E(C) \cup \bigcup_{\substack{j \in [g]\\v \in W_{i,j}}} \{c_j, v\}.
\end{align*}
Furthermore, let $V_1, \ldots, V_g \subseteq D'$ be arbitrarily chosen disjoint subsets of order $\varepsilon n / (16 \log^2 n)$ (this is possible since $|D'| \ge \varepsilon n / (8 \log n)$), define the family of canonical cycles $\mathcal{C}$ as
$$ \mathcal{C} := \{ C = (c_1, \ldots, c_g) \mid C \; \text{is a cycle and} \; c_j \in V_j \; \text{for all} \; j \in [g] \}$$
and set
$$ \mathcal{C}^+ := \{ C \oplus i \mid C \in \mathcal{C} \; \text{and} \;  i \in [k]\}. $$
Observe that if $G$ contains any graph from $\mathcal{C}^+$, then $G$ contains the desired cycle. Using Janson's inequality, we upper bound the probability that this does not happen. In the remainder of the proof, we will estimate the parameters $\mu$ and $\delta$ defined in Theorem \ref{thm:Janson}.

Note that each graph $C^+ \in \mathcal{C}^+$ appears in $G$ with probability $p^{g + (\Delta - 2)g} = p^{(\Delta - 1)g}$. Therefore,
$$\mu = |\mathcal{C}^+|  p^{(\Delta-1)g} = k \left(\frac{\varepsilon n}{16 \log^2 n} \right)^g  p^{(\Delta-1)g} \gg k \log^3 n.$$
Next, we wish to show that
$\delta = o(\mu^2/k \log^3 n)$. By definition, we have
$$ \delta = \sum_{i, j \in [k]} \sum_{\substack{C', C'' \in \mathcal{C}\\C' \oplus i \sim C'' \oplus j}} p^{e(C' \oplus i) + e(C'' \oplus j) - e((C' \oplus i) \cap (C'' \oplus j))}. $$
We consider the cases $i \neq j$ and $i = j$ separately.

First, if $C' \oplus i \sim C'' \oplus j$ for $i \neq j$ and $C',
C'' \in \mathcal{C}$, then we have $(C' \oplus
i) \cap (C'' \oplus j) = C' \cap C''$. Let $J := C' \cap C''$ and
observe that $e(J) \ge 1$, as otherwise $C' \oplus i$ and $C'' \oplus
j$ would not have any edges in common. Let $\mathcal{J}_1$ be the
family consisting of all possible graphs of the form $C' \cap C''$,
$$ \mathcal{J}_1 := \{J = C' \cap C'' \mid C', C'' \in \mathcal{C} \; \text{and} \; e(J) \ge 1\}. $$
We can now estimate the contribution of such pairs to $\delta$ as follows:
\begin{align*}
\delta_1 &=  \sum_{i \neq j} \sum_{\substack{C', C'' \in \mathcal{C}\\e(C' \cap C'') \ge 1}}  p^{e(C' \oplus i) + e(C'' \oplus j) - e(C' \cap C'')} =  \sum_{i \neq j} \sum_{J \in \mathcal{J}_1} \sum_{\substack{C', C'' \in \mathcal{C}\\C' \cap C'' = J}}  p^{2 (\Delta - 1)g  - e(J)} \\
&\le k^2 \sum_{J \in \mathcal{J}_1} \left( \frac{\varepsilon n}{16 \log^2 n} \right)^{2(g - v(J))} p^{2(\Delta-1)g - e(J)} = \mu^2 \sum_{J \in \mathcal{J}_1} \left( \frac{\varepsilon n}{16 \log^2 n} \right)^{-2 v(J)} p^{-e(J)}.
\end{align*}
Since  $e(J) = 1$ for $v(J) = 2$ and $ e(J) \le v(J)$  otherwise, we can bound the last sum by
\begin{align*}
\sum_{J \in \mathcal{J}_1} \left( \frac{\varepsilon n}{16 \log^2 n} \right)^{-2 v(J)} p^{-e(J)} \le \sum_{\substack{J \in \mathcal{J}_1\\v(J) = 2}} \left( \frac{\varepsilon n}{16 \log^2 n} \right)^{-4} p^{-1} +  \sum_{\substack{J \in \mathcal{J}_1\\v(J) > 2}} \left( \frac{\varepsilon n}{16 \log^2 n} \right)^{-2 v(J)} p^{-v(J)}.
\end{align*}
Observe that there are at most $\binom{g}{v_J} \left( \eps n / (16 \log^2 n) \right)^{v_J}$ graphs $J \in \mathcal{J}_1$ on $v_J$ vertices. Thus, we have
\begin{multline*}
\sum_{J \in \mathcal{J}_1} \left( \frac{\varepsilon n}{16 \log^2 n} \right)^{-2 v(J)} p^{-e(J)} \le \binom{g}{2} \left( \frac{\varepsilon n}{16 \log^2 n} \right)^{-2} p^{-1} + \sum_{v_J > 2} \binom{g}{v_J} \left(\frac{\eps n}{16 \log^2 n} \right)^{-v_J} p^{-v_J} \\
 \ll (2\log n)^2 \left( \frac{16 \log^2 n}{\eps n} \right)^2 n^{1/(\Delta-1)} + \sum_{v_J > 2} \binom{2 \log n}{v_J} \left(\frac{\eps n}{16 \log^2 n} \cdot \frac{1}{n^{1/(\Delta-1)}} \right)^{-v_J} \ll \frac{32}{\varepsilon n} \le \frac{1}{k \log^3 n},
\end{multline*}
where we used that $\Delta \geq 3$. Therefore, we obtain $\delta_1 = o(\mu^2 / k \log^3 n)$.

Let us now consider the case $C' \oplus i \sim C'' \oplus i$, for some $i \in [k]$ and distinct cycles $C', C'' \in \mathcal{C}$. Let $J := C' \cap C''$ and observe that $v(J) \ge 1$ and $v(J) \le g - 1$. As before, let $\mathcal{J}_2$ be the family consisting of all possible graphs of the form $C' \cap C''$,
$$ \mathcal{J}_2 := \{J = C' \cap C'' \mid C', C'' \in \mathcal{C} \; \text{and} \; v(J) \in \{1, \ldots, g - 1\}\}. $$
Note that if $V(J) \cap V_q = \{v\}$ for some $q \in \{1, \ldots, g\}$, then $\{v, w\} \in E(C' \oplus i \cap C'' \oplus i)$ for all $w \in W_{i, q}$. Therefore, we have
$$
	e(C' \oplus i \cap C'' \oplus i) = e(J) + v(J)(\Delta - 2).
$$
With these observations in hand, we can bound the contribution of such pairs to $\delta$ as follows:
\begin{multline*}
\delta_2   = \sum_{i \in [k]} \sum_{\substack{C', C'' \in \mathcal{C}\\v(C' \cap C'') \ge 1}}  p^{e(C' \oplus i) + e(C'' \oplus i) - e(C' \oplus i \cap C'' \oplus i)}
 = \sum_{i \in [k]} \sum_{J \in \mathcal{J}_2} \sum_{\substack{C', C'' \in \mathcal{C}\\C' \cap C'' = J}} p^{2(\Delta-1)g - (e(J) + v(J)(\Delta - 2))} \\
\le k \sum_{J \in \mathcal{J}_2} \left(\frac{\varepsilon n}{16 \log^2 n} \right)^{2(g - v(J))}  p^{2(\Delta-1)g - (e(J) + v(J)(\Delta-2))}
 \le \frac{\mu^2}{k} \sum_{J \in \mathcal{J}_2}  \left(\frac{\varepsilon n}{16 \log^2 n} \right)^{-2 v(J)}  p^{-(e(J) + v(J)(\Delta - 2))}.
\end{multline*}
Since $J$ is a subgraph of a cycle, we have $e(J) \leq v(J)$. Therefore, we can bound the last sum by
$$
\sum_{J \in \mathcal{J}_2}  \left(\frac{\varepsilon n}{16 \log^2 n} \right)^{-2 v(J)}  p^{-(e(J) + v(J)(\Delta - 2))} \le \sum_{J \in \mathcal{J}_2} \left(\frac{\varepsilon n}{16 \log^2 n} \right)^{-2 v(J)}  p^{-v(J)(\Delta - 1)}.
$$
Note that there are at most $\binom{g}{v_J} \left( \eps n / (16 \log^2 n) \right)^{v_J}$ graphs $J \in \mathcal{J}_2$ on $v_J$ vertices. Thus, we have
\begin{align*}
\sum_{J \in \mathcal{J}_2}  \left(\frac{\varepsilon n}{16 \log^2 n} \right)^{-2 v(J)}  p^{-v(J)(\Delta - 1)} &\le \sum_{v_J = 1}^{g-1} \binom{g}{v_J} \left( \frac{\eps n}{16 \log^2 n} \right)^{-v_J} p^{-v(J)(\Delta - 1)} \\
&\ll \sum_{v_J = 1}^{g-1} \left( 2 \log n \cdot \frac{16 \log^2 n}{\eps n} \cdot \frac{n}{\log^7 n} \right)^{v_J} \le \frac{1}{\log^3 n}.
\end{align*}
Therefore, we obtain $\delta_2 = o(\mu^2 / (k \log^3 n))$.

Finally, we have $\delta \le \delta_1 + \delta_2 = o(\mu^2/(k \log^3 n))$ and Theorem \ref{thm:Janson} gives the desired upper bound on the probability that $G$ does not contain any graph from $\mathcal{C}^+$, completing the proof of the lemma.
\end{proof}

\begin{proof}[Proof of Lemma \ref{lemma:embedding_cycles}]
Let $G \sim \Gnp$ with $p$ as stated in the lemma. For each $i \in [t]$, we define a $g$-uniform hypergraph $H_i := (D, E_i)$ as follows: a set of vertices $\{v_1, \ldots, v_g\} \subseteq D$ forms a hyperedge if and only if $G[\{v_1, \ldots, v_g\}]$ contains a cycle $C = (c_{1}, \ldots, c_{g})$ such that $W_{i, j} \subseteq \Gamma_G(c_{j})$ for all $j \in [g]$. Observe that the existence of a function $f$ with properties as in Theorem \ref{thm:hyper_match}, applied to the family  $\mathcal H := \{ H_1, \ldots, H_t\}$, implies the existence of the desired family of cycles. Therefore, it is sufficient to prove that, with probability at least $1 - o(1 / n)$, for every $\mathcal I \subseteq [t]$ the hypergraph $\bigcup_{i \in \mathcal{I}} E(H_i)$ contains a matching of size at least $g |\mathcal I |$.

Since the requirements on the family $\{W_{i,j}\}_{(i,j)\in[t]\times[g]}$ are the same as in Lemma \ref{lemma:cycles}, it follows from the union bound that, with probability at least $1 - o(1/n)$, $G$ satisfies the property given by Lemma \ref{lemma:cycles} for $D$, $g$ and every $k \in [t]$. Let $\mathcal I \subseteq [t]$ be a subset of size $k$ for some $k \in [t]$ and set $D' := D$. We now apply the following procedure $k g$ times: using the property given by Lemma \ref{lemma:cycles} for $D'$, there exists a hyperedge $e \in \bigcup_{i \in \mathcal{I}} E(H_i)$ such that $e \subset D'$, and set $D' := D' \setminus e$. Since $|e| = g$ and we repeat the procedure $kg$ times, we have $|D'| \ge |D| - kg^2$ throughout the whole process. Thus, we can indeed apply Lemma \ref{lemma:cycles} in each step. Furthermore, since every edge is vertex disjoint from the previously obtained edges, we have constructed a matching in $\bigcup_{i \in \mathcal{I}} E(H_i)$ of size $k g$. As $\mathcal{I} \subseteq [t]$ was chosen arbitrarily, this concludes the proof of the lemma.
\end{proof}

\section{Proof of Theorem \ref{thm:main_universal}}

Our proof strategy goes as follows. Given a graph $H \in
\mathcal{H}((1 - \eps)n, \Delta)$, we first remove vertices belonging to small
connected components from $H$, writing $H_1$ for the resulting
graph. Working in $H_1$, we then remove carefully chosen induced
cycles of length at most $2 \log n$ (again, we remove vertices) in such a way that the
resulting graph $H_2$ belongs to the family of graphs
$\mathcal{F}((1 - \eps')n, \Theta(\log^3 n), \eps'', \Delta - 1)$,
for some parameters $\eps'$ and $\eps''$ tending to zero with $\eps$. Now, using Theorem
\ref{thm:universality_partition}, we find an embedding of $H_2$.
Then, using Lemma \ref{lemma:embedding_cycles}, we place the
removed cycles into $G$ in an appropriate way. Finally, using Lemma \ref{lemma:small_universal}, we
complete the embedding of $H$ by embedding small components one by
one. We will now give a formal description of this procedure.

\vspace{3mm}
\noindent
\textbf{Preparing the graph $G$.} Fix some $\eps > 0$ and integer $\Delta \ge 3$. Let $R, D_3, \ldots, D_{2\log n} \subseteq [n]$ be arbitrarily chosen disjoint subsets of $\{1, \ldots, n\}$ such that $|R| = (1 - \eps/2)n$ and $|D_i| = \eps n/ (4 \log n)$ for each $i \in \{3, \ldots, 2 \log n\}$.
Let $G$ be a graph with the following properties:
\begin{enumerate}[(i)]
\item the induced subgraph $G[R]$ is $\mathcal{F}((1 - \eps/2)n, (\Delta^2 +1)q + 1, \eps', \Delta - 1)$-universal, where $q = 65 \eps^{-1} \log^3 n $ and $\eps' = \min\{1/(2\Delta), \eps/(2 - \eps)\}$,

\item for every subset $V' \subseteq V(G)$ of order $|V'| \ge \eps n$, the induced subgraph $G[V']$
contains every connected graph from the family $\mathcal{H}(\log^4 n, \Delta)$, and

\item $G$  satisfies the property given by Lemma \ref{lemma:embedding_cycles} for every $g \in \{3, \ldots, 2 \log n\}$, $t \le \varepsilon n / (32 \log^3 n)$ and $D = D_g$.
\end{enumerate}
Observe that by Theorem \ref{thm:universality_partition} and Lemmas \ref{lemma:small_universal} and \ref{lemma:embedding_cycles}, $G \sim \Gnp$ satisfies properties (i)--(iii) asymptotically almost surely, provided that $p \gg n^{-1/(\Delta - 1)} \log^5 n$. We remark that the bound on $p$ here is determined by Theorem \ref{thm:universality_partition}.

\vspace{3mm}
\noindent
\textbf{Preparing the graph $H$.} Let $H \in \mathcal{H}((1 - \eps)n, \Delta)$ and let $H_1 \subseteq H$ be the subgraph which consists of all connected components of $H$ with at least $\log^4 n$ vertices. The following observation plays a crucial role in our argument.

\begin{claim} \label{claim:crucial}
For every vertex $v \in H_1$, at least one of the following properties hold:
\begin{enumerate}[$(a)$]
\item $B_{H_1}^{(\log n)}(v)$ contains a vertex $w$ with $\deg_{H_1}(w) \le \Delta - 1$, or
\item $H_1[B_{H_1}^{(\log n)}(v)]$ contains a cycle of length at most $2 \log n$.
\end{enumerate}
\end{claim}
\begin{proof}
Let us assume the opposite, i.e., for every vertex $w \in B_{H_1}^{(\log n)}(v)$ we have $\deg_{H_1}(w) = \Delta$ and $H_1[B_{H_1}^{(\log n)}(v)]$ contains no cycle of length at most $2\log n$. Then $H_1[B_{H_1}^{(\log n)}(v)]$ is a tree and, since $\Delta \ge 3$, it contains at least $\sum_{j = 1}^{\log n} (\Delta-1)^{j} > n$ vertices, which is clearly a contradiction.
\end{proof}

Let $I \subseteq V(H_1)$ be a maximal $(64 \eps^{-1} \log^3 n )$-independent set in $H_1$. Write $I_a$ for the set of all vertices in $I$ which satisfy property $(a)$ of Claim \ref{claim:crucial} and set $I_b := I \setminus I_a$. Furthermore, for each $v \in I_b$, let $C_v$ be a cycle of smallest length in $H_1[B_{H_1}^{(\log n)}(v)]$, let $\ell_v$ denote its length and fix an arbitrary ordering $(c_v^1, \ldots, c_v^{\ell_v})$ of the vertices along $C_v$. By minimality, $C_v$ is an induced cycle. Finally, let $H_2  := H_1 \setminus \left[ \bigcup_{v \in I_b} V(C_v) \right]$ and note that the $(64\eps^{-1} \log^3n)$-independence of $I$ implies
\begin{equation} \label{eq:removed_cycle}
B_{H_1}^{(3 \log n)}(v) \cap V(H_2) = \begin{cases}
    B_{H_1}^{(3 \log n)}(v), &\text{for $v \in I_a$},\\
    B_{H_1}^{(3 \log n)}(v) \setminus V(C_v), &\text{for $v \in I_b$}.
    \end{cases}
\end{equation}

\vspace{3mm}
\noindent
\textbf{Phase I: Embedding $H_2$ into $G[R]$.} We claim that there exists an embedding of $H_2$ into $G[R]$. Let $H_2'$ be a graph on $(1 - \eps/2)n$ vertices obtained from $H_2$ by adding isolated vertices. Using property (i) of the graph $G$, in order to show that there exists an embedding of $H_2$ into $G[R]$, it will suffice to prove that $H_2' \in \mathcal{F}((1 - \eps/2)n, (\Delta^2 +1) q + 1, \eps', \Delta - 1)$, where
$q = 65 \eps^{-1} \log^3 n  $ and $\eps' = \min\{1/2\Delta, \eps/(2 - \eps)\}$. We prove this by finding a partition
$W_0, W_1, \ldots, W_{(\Delta^2 + 1) q + 1}$ of $V(H_2')$ with the following properties:
\begin{enumerate}[(i)]
\item $|W_{(\Delta^2 + 1) q + 1}| = \lfloor \eps'(1 - \eps /2 )n \rfloor$,
\item $W_0 = \Gamma_{H_2'}(W_{(\Delta^2 + 1) q + 1})$,
\item $W_{(\Delta^2 + 1) q + 1}$ is $3$-independent (in $H_2'$),
\item $W_i$ is $2$-independent (in $H_2'$) for every $1\leq i \leq (\Delta^2 +1) q$, and
\item for every $1\leq i\leq (\Delta^2 +1) q + 1$ and for every $w\in W_i$, $w$ has at
most $\Delta - 1$ neighbors in $W_0\cup\ldots \cup W_{i-1}$.
\end{enumerate}

First, note that $H_2'$ contains at least $\eps n / 2$ isolated vertices as $|V(H_2)| \le (1- \eps)n$. Since $\eps' \le \eps/(2 - \eps)$ or equivalently $\eps' (1 - \eps / 2) \le \eps / 2$, we can set $W_{(\Delta^2 + 1)  q + 1}$ to be a set of $\eps' (1 - \eps / 2)  n$ isolated vertices. Then $W_0 = \emptyset$ and $W_{(\Delta^2 +1) q + 1}$ is trivially $3$-independent. Furthermore, let $S_{q} \subseteq V(H_2') \setminus W_{ (\Delta^2  + 1) q + 1}$ be the set of all remaining vertices in $H_2'$ with degree at most $\Delta - 1$ in $H_2$ and observe that for each $v \in I$ we have
\begin{equation}
\label{eq:sq_i}
S_q \cap B_{H_1}^{(3 \log n)}(v) \neq \emptyset.
\end{equation}
For $v \in I_a$, this follows from \eqref{eq:removed_cycle} and the definition of the set $I_a$. For $v \in I_b$, we have from \eqref{eq:removed_cycle} and $|B_{H_1}^{(3\log n)}(v)| \ge 3 \log n > |V(C_v)|$ that $B_{H_1}^{(3 \log n)}(v) \cap V(H_2) \neq \emptyset$. Thus there exists a vertex $w \in B_{H_1}^{(3 \log n)}(v) \cap V(H_2)$ adjacent to some vertex in $C_v$, and clearly $\deg_{H_2}(w) \le \Delta - 1$.

Next, for each $i \in \{1, \ldots, q-1\}$, we define
$$
S_{q-i} := \Gamma_{H_2}^{(i)}(S_q).
$$
We first show that $S_1,  \ldots, S_q, W_{(\Delta^2 +1)q +1}$ is a partition of $V(H_2')$. Since disjointness follows from the construction, it suffices to prove that for each $w \in V(H_2') \setminus  W_{(\Delta^2 + 1) q +1}$ we have
$B^{(q-1)}_{H_2}(w) \cap S_q \neq \emptyset$.
 This can be seen as follows. Since $I$ is a maximal $(64 \eps^{-1} \log^3 n)$-independent set in $H_1$, for each vertex $w \in V(H_2) \setminus I$ we have $B_{H_1}^{(64 \eps^{-1}\log^3 n)} (w) \cap I \neq \emptyset$. Otherwise, we could extend $I$, contradicting its maximality. Thus, from \eqref{eq:sq_i}, we conclude that for each  $w \in V(H_2) \setminus I$
we have $B_{H_1}^{(q - 1)} (w) \cap S_q \neq \emptyset$. Let us now consider the shortest path in $H_1$ from $w$ to a vertex $s \in S_q$, and denote the vertices along such a path by $w = p_0, p_1, p_2, \ldots, p_{q'} = s$, for some $q' \le q - 1$. If $p_1, \ldots, p_{q'} \in V(H_2)$, then clearly $s \in B_{H_2}^{(q-1)}(w)$. Otherwise, let $i'$ be the smallest index such that $p_{i'} \notin V(H_2)$. But then $\deg_{H_2}(p_{i'-1}) \le \Delta - 1$ and thus, by definition, $p_{i' - 1} \in S_q$, which again implies $B_{H_2}^{(q-1)}(w) \cap S_q \neq \emptyset$. This shows that $S_1, \ldots, S_q, W_{(\Delta^2 + 1)q +1}$ is indeed a partition of $V(H_2')$.

Furthermore, by construction, for each $i \in \{1, \ldots, q - 1\}$ and each vertex $v \in S_i$, $v$ has at least one neighbor in $\bigcup_{j = i+1}^{q} S_j$ and thus at most $\Delta - 1$ neighbors in $\bigcup_{j = 1}^{i-1} S_j$.
However, the sets $S_i$ are not necessarily $2$-independent in $H_2'$. This can be fixed in the following way. The square of $H_2'$, denoted by $(H_2')^2$, has maximum degree at most $\Delta^2$. Therefore, $(H_2')^2$
 can be partitioned into $\Delta^2 + 1$ sets $L_1, \ldots, L_{\Delta^2 + 1}$ which are independent in $(H_2')^2$ and thus $2$-independent in $H_2'$. Now, by setting $W_{(i - 1) (\Delta^2 +1) + j} := S_i \cap L_j$ for every $i \in \{1, \ldots, q\}$ and $j \in \{1, \ldots, \Delta^2 + 1\}$, we obtain a partition of $V(H_2')$ satisfying properties $(i)$--$(v)$.

To conclude, we have shown that $H_2' \in \mathcal{F}((1 - \eps/2)n, (\Delta^2 +1) q + 1, \eps', \Delta - 1)$. Thus, by property $(i)$ of the graph $G$, there exists an embedding $f:V(H_2) \rightarrow R$ of $H_2$ into $G[R]$.

\vspace{3mm}
\noindent
\textbf{Phase II: Embedding removed cycles.} Consider some $g \in \{3, \ldots, 2 \log n\}$ and let $I_g \subseteq I_b$ be the set of all vertices $v \in I_b$ such that $\ell_v = g$. For each $(v, j) \in I_g \times [g]$, let $W_{v,j} := f(\Gamma_H(c_v^j) \cap V(H_2))$. Note that, by construction, the family of subsets $\{W_{v,j}\}_{(v,j) \in I_g \times [g]}$ satisfies requirements $(i)$ and $(ii)$ of Lemma \ref{lemma:embedding_cycles} with $D = D_g$. Thus, in order to apply Lemma \ref{lemma:embedding_cycles}, it suffices to show that $I_g$ is not too large. The following claim provides an upper bound on $I$, and thus on $I_g$, which in this case suffices.

\begin{claim}
$|I| \le \frac{\varepsilon n}{32\log^3 n}.$
\end{claim}
\begin{proof}
Let $u, v \in I$ be two distinct vertices from $I$. Then $B_{H_1}^{(32 \eps^{-1}\log^3 n)}(u) \cap
B_{H_1}^{(32 \eps^{-1} \log^3 n)}(v) = \emptyset$, as otherwise we would have
$\mathrm{dist}(u, v) \le 64 \eps^{-1} \log^3 n $ and the set $I$ would not be $(64 \eps^{-1} \log^3 n )$-independent. On the other hand, for every $v \in I$ we have that either $\Gamma_{H_1}^{(i)}(v) \neq \emptyset$ for every $i \in  \{1, \ldots, 32 \eps^{-1} \log^3 n  \}$ or $B_{H_1}^{(32 \eps^{-1} \log^3 n )}(v)$ contains the whole connected component of $v$, which is of order at least $\log^4 n$. In either case, we have $|B_{H_1}^{(32 \eps^{-1} \log^3 n )}(v)| \ge 32 \eps^{-1} \log^3 n $ and the desired bound on $I$ follows.
\end{proof}

Therefore, by property (iii) of the graph $G$, there exists a family $\{(c_{v, 1}, \ldots, c_{v,g})\}_{v \in I_g}$ of vertex disjoint cycles in $G[D_g]$ such that setting $f(c_v^j) := c_{v, j}$ for every $(v, j) \in I_g \times [g]$ defines an embedding of $H_2 \cup \left[ \bigcup_{v \in I_g} C_v \right]$ into $G[R \cup D_g]$. Since this holds for every $3 \le g \le 2 \log n$ and the sets $D_3, \ldots, D_{2 \log n}$ are disjoint, we obtain an embedding of $H_1$ into $G$.

It is worth remarking that the bound on $I$, which facilitates the application of Lemma \ref{lemma:embedding_cycles}, is the reason why we treat small connected components separately.

\vspace{3mm}
\noindent
\textbf{Phase III: Embedding small components.} As a last step, we have to extend our embedding of $H_1$ to an embedding of the whole graph $H$. Using the facts that $H$ is of order $(1 - \eps)n$ and each component of $H$ which is not in $H_1$ is of order at most $\log^4 n $, we can greedily embed these components one by one as follows. Consider one such component and let $V' \subseteq V(G)$ be the set of vertices which are not an image of some already embedded vertex of $H$. Then $|V'| \ge \eps n$ and, by property $(ii)$ of the graph $G$, $G[V']$ contains an embedding of the required component. Repeating the same argument for each component which has not yet been embedded, we obtain an embedding of the graph $H$.

\section{Concluding remarks}

The observant reader will have noticed that our argument does not apply when $\Delta = 2$. In this case, one cannot hope to show that universality holds all the way down to $p \approx n^{-1/(\Delta - 1)} = n^{-1}$. Even to find a collection of $(1 - \varepsilon) \frac{n}{3}$ disjoint triangles, the probability must be at least $n^{-2/3}$. Since every graph with maximum degree $2$ is a disjoint union of paths and cycles, it is not too hard to use arguments similar to those in Section~\ref{sec:components} to show that for any $\varepsilon > 0$ there exists a constant $c >0$ such that if $p \geq c n^{-2/3}$, the random graph $G(n,p)$ is a.a.s.~$\mathcal H((1- \varepsilon) n, 2)$-universal.

Our proof relies heavily on the fact that the graphs we are hoping to embed are almost spanning rather than spanning. In particular, we neither know how to complete the removed cycles nor how to add small components back into the graph without making heavy use of the almost-spanning condition. Given this, it seems likely that a spanning analogue of Theorem~\ref{thm:main_universal} will require new ideas.

We have already mentioned that a modification of the embedding techniques from Alon et al.~\cite{ACKRRS} was used by Kohayakawa, R\"odl, Schacht and Szemer\'edi \cite{KRSS} to prove a subquadratic bound on the size Ramsey number of bounded-degree graphs. It would be interesting to know whether our embedding technique might be useful for improving this bound. 

\vspace{3mm}
\noindent
{\bf Acknowledgements.} We would like to thank the anonymous referees for their thorough reviews and valuable remarks.

\bibliographystyle{abbrv}
\bibliography{references}

\end{document}